\documentclass[10pt]{article}

\usepackage{amsfonts,amsmath,amsthm}
\usepackage{amssymb}
\usepackage{xcolor}
\usepackage{graphicx}
\usepackage{stmaryrd}        
\usepackage{multirow}
\usepackage{dsfont}
\usepackage{rotating}
\usepackage{algorithm}
\usepackage{algpseudocode}

\usepackage[paper=a4paper,dvips,top=2cm,left=2cm,right=2cm,
foot=1cm,bottom=4cm]{geometry}

\newcommand{\blue}[1]{#1}

\newtheorem{Thm}{Theorem}[section]
\newtheorem{Def}[Thm]{Definition}

\newtheorem{Prop}[Thm]{Proposition}

\newtheorem{example}[Thm]{Example}
\newtheorem{remark}[Thm]{Remark}

\begin{document}
	
\title{\blue{Optimization Models and Computational Bounds for Limited Augmented Zarankiewicz Numbers in the Incidence-Graph Family of Complete Graphs}}
	\author{Yi Xu\footnote{School of Mathematics, Southeast University, Nanjing 211189, China;  Jiangsu Provincial Scientific Research Center of Applied Mathematics, Nanjing 211189, China. ({\tt yi.xu1983@hotmail.com})}, \  Gaohang Yu\thanks{School of Science, Zhejiang University of Science and Technology.(E-mail:{\tt maghyu@163.com})}, \ Liqun Qi\thanks{Jiangsu Provincial Scientific Research Center of Applied Mathematics, Nanjing 211189, China; Department of Applied Mathematics, The Hong Kong Polytechnic University, Hung Hom, Kowloon, Hong Kong. ({\tt maqilq@polyu.edu.hk})}
	}
	
	\date{}
	\maketitle

\begin{abstract}
{Let $G_1$ denote the incidence graph of the complete graph $K_{q+1}$. We study limited augmented Zarankiewicz numbers in this family through structured 0--1 optimization models for admissible 2-edge augmentation. For the original framework, we combine exact ILP formulations for the smallest instances with constructive search followed by exact admissibility verification for larger instances. This yields}
\[
{z_L(6,4)=14,\qquad z_L(10,5)=26,\qquad z_L(15,6)\ge 43,\qquad z_L(21,7)\ge 64,\qquad z_L(28,8)\ge 88.}
\]
{The first two values are exact, whereas the latter three are rigorous lower bounds obtained from explicitly verified admissible families.}

{We then formulate the global weak problem determined by \((S),(W2),(W2'),(W3)\). On the larger incidence-family cases, the corresponding global exact optimization problem is currently computationally out of reach, and our computations provide the certified lower bounds}
\[
{z_{WL}(6,4)=14,\quad z_{WL}(10,5)\ge 29,\quad z_{WL}(15,6)\ge 46,\quad z_{WL}(21,7)\ge 65,\quad z_{WL}(28,8)\ge 90.}
\]
{On the fourteen literal-grid benchmarks from $(4,3)$ through $(7,7)$, by contrast, the global weak problem is solved exactly. Across this exact benchmark block, the weak framework is never worse than the strong framework and is strictly better in several cases. These computations improve the corresponding classical Zarankiewicz numbers and therefore strengthen available lower bounds for $\operatorname{BSR}(m,n)$ within this family.}


\end{abstract}

\section{Introduction}

Let \(m \ge n \ge 2\). The \textbf{Zarankiewicz number} \(z(m,n)\) is the maximum number of edges in an \(m \times n\) bipartite graph that contains no complete bipartite subgraph \(K_{2,2}\) (equivalently, no cycle \(C_4\)). This classical extremal problem goes back to Zarankiewicz \cite{Za51}; a general upper bound was proved by K\"{o}v\'{a}ri, S\"{o}s and Tur\'{a}n \cite{KST54}. Also see \cite{CHM24}.

A basic extremal family is obtained from complete graphs. Let \(V(K_{q+1})=\{0,1,\dots,q\}\). The incidence graph of \(K_{q+1}\) has left vertex set \(\binom{V(K_{q+1})}{2}\) and right vertex set \(V(K_{q+1})\). It contains one edge for each incidence of a graph edge with one of its endpoints. This graph is \(C_4\)-free and has \(\binom{q+1}{2}\) left vertices, \(q+1\) right vertices, and \(q(q+1)\) edges. It is classical that
\[
z\!\left(\binom{q+1}{2},q+1\right)=q(q+1),
\]
see Reiman \cite{Re58} and Guy \cite{Gu69}. In particular,
\[
\begin{aligned}
z(6,4)&=12,\qquad z(10,5)=20,\qquad z(15,6)=30,\\
z(21,7)&=42,\qquad z(28,8)=56.
\end{aligned}
\]

Qi, Cui, and Xu \cite{QCX26a} introduced the \textbf{augmented Zarankiewicz number} \(z_A(m,n)\) and the \textbf{limited augmented Zarankiewicz number} \(z_L(m,n)\). Both parameters arise from the study of SOS ranks of biquadratic forms. The 1-edge subgraph is denoted by
\[
G_1=(S,T,E_1).
\]
For \(z_A(m,n)\), the graph \(G_1\) may be any extremal \(C_4\)-free bipartite graph. For \(z_L(m,n)\), one requires \(|E_1|=z(m,n)\). In both cases one allows the addition of 2-edges, subject to constraints forbidding generalized \(C_4\)-configurations.

{These augmentation parameters are motivated by the SOS rank of biquadratic forms. In brief, \(\operatorname{BSR}(m,n)\) \cite{QCX26a} is the maximum SOS rank among all \(m\times n\) sum-of-squares biquadratic forms. The key link proved in \cite{QCX26a} is that}
\[
{\operatorname{BSR}(m,n) \ge z_L(m,n) \ge z(m,n),}
\]
{so every lower bound for \(z_L(m,n)\) immediately yields a lower bound for \(\operatorname{BSR}(m,n)\). This is the optimization-theoretic reason for studying admissible 2-edge augmentations of extremal \(C_4\)-free graphs.}

{From this viewpoint, the strong framework fixes an extremal 1-edge graph \(G_1\) and seeks a largest compatible family \(E_2\) of 2-edges, while the weak framework later considered in the paper optimizes globally over both \(E_1\) and \(E_2\). Thus the strong framework provides the baseline model, and the weak framework provides a larger comparison class.}

A companion paper \cite{QCX26b} established a general lower bound by means of a \(K_{4t}\) block construction. If \(m=\binom{4t}{2}\) and \(n=4t\), then
\[
z_L(m,n)\ge 2\binom{4t}{2}+4t^2-2t.
\]
Consequently,
\[
\frac{z_L-z}{z}\ge \frac{4t^2}{16t^2-4t}\longrightarrow \frac14,
\]
so the gap is asymptotically at least \(25\%\) of the classical Zarankiewicz number.

{In the present paper we focus on the incidence-graph family arising from complete graphs, because it is structured enough to admit exact ILP models and still rich enough to exhibit a visible gap between the classical, strong, and weak parameters. We begin in Section~2 with a miniature \(q=2\) example that makes the row, column, cell, and 2-edge notation concrete before turning to the general theory.}

In this paper we focus on the incidence-graph family arising from complete graphs. This family is highly structured, the 1-edge subgraph is unique up to isomorphism, and the admissibility conditions admit finite optimization models with a strong built-in symmetry. This makes the family a natural testbed for combining exact 0--1 formulations, exact solution of small benchmark instances, and certified lower-bound procedures for larger instances. Our goal is to determine or bound
\[
z_L\!\left(\binom{q+1}{2},q+1\right)
\]
and its weak counterpart within this family for small values of \(q\), and to compare the original and weak frameworks on common exact and incidence-family benchmark sets.

\subsection*{Main contributions}

The main contributions are as follows.
\begin{enumerate}
  \item We formulate the original limited augmented Zarankiewicz problem in the incidence-graph family as a structured 0--1 optimization model with occupancy variables and explicit linear encodings of \((S)\), \((C2)\), and \((C3)\), with the nondegeneracy restriction \((C1)\) enforced at the candidate-family level.
  \item \blue{For the larger strong-framework instances, we complement the exact model by a constructive search procedure based on randomized greedy insertion, {local delete-and-refill improvement}, and final exact admissibility verification, thereby producing certified feasible solutions and rigorous lower bounds.}
  \item We obtain the exact values
  \[
  z_L(6,4)=14,\qquad z_L(10,5)=26.
  \]
  \item We obtain the lower bounds
  \[
  z_L(15,6)\ge 43,\qquad z_L(21,7)\ge 64,\qquad z_L(28,8)\ge 88.
  \]
  \blue{These bounds come from directly verified admissible families with \(|E_2|=13\), \(|E_2|=22\), and \(|E_2|=32\), respectively.}
  \item \blue{We formulate the global weak problem determined by \((S),(W2),(W2'),(W3)\) and develop the corresponding exact and heuristic computational procedures.}
  \item \blue{Within the incidence-graph family, the global exact weak problem is tractable at \((6,4)\), while for \((10,5)\), \((15,6)\), \((21,7)\), and \((28,8)\) the present computations provide the rigorous lower bounds \(z_{WL}(6,4)=14\), \(z_{WL}(10,5)\ge 29\), \(z_{WL}(15,6)\ge 46\), \(z_{WL}(21,7)\ge 65\), and \(z_{WL}(28,8)\ge 90\).}
  \item \blue{On the fourteen exact literal-grid benchmarks from \((4,3)\) through \((7,7)\), the global weak problem is solved exactly, and the weak framework is never worse than the strong one and is strictly better in several cases.}
  \item We compare the resulting values and bounds with the \(K_{4t}\) construction and observe that the gap ratios in the computed small cases already exceed the asymptotic \(25\%\) benchmark.
\end{enumerate}

\subsection*{Outline}

{Section 2 begins with a small \(q=2\) illustration and then introduces the general incidence-graph notation, the degenerate/nondegenerate distinction, and the admissibility rules.} \blue{Section 3 develops the strong-framework optimization model and the associated computational procedure, including candidate-family generation, direct ILP, and constructive search.} {Section 4 then reports the resulting exact values and certified lower bounds in that baseline setting.} \blue{Section 5 develops the weak-framework optimization models and computational procedures, presenting the global definition of \(z_{WL}\), the exact global literal-grid benchmarks, and the current incidence-family lower bounds for the larger cases.} {Section 6 compares the strong and weak outcomes on common benchmark sets, so that the effect of relaxing the admissibility conditions can be read directly from the tables.} \blue{Section 7 records a lifting comparison for \((21,7)\).} Section 8 compares the resulting ratios with the \(K_{4t}\) construction, and Section 9 concludes with a brief list of open problems.

\section{Incidence-graph model and admissibility structure for \(K_{q+1}\)}

Fix an integer \(q\ge 2\). Let
\[
V=\{0,1,\dots,q\},\qquad S=\binom{V}{2},\qquad T=V.
\]
Define
\[
E_1=\{(e,v)\in S\times T: v\in e\}.
\]
Then
\[
G_1=(S,T,E_1)
\]
is the incidence graph of \(K_{q+1}\). It has
\[
\begin{aligned}
|S|&=\binom{q+1}{2},\qquad |T|=q+1,\\
|E_1|&=2\binom{q+1}{2}=q(q+1)=z\!\left(\binom{q+1}{2},q+1\right).
\end{aligned}
\]

\begin{example}
{The smallest nontrivial incidence-family case is \(q=2\). Then \(V=\{0,1,2\}\),}
\[
{S=\{01,02,12\},\qquad T=\{0,1,2\},}
\]
{and the incidence 1-edge set is}
\[
{E_1=\{(01,0),(01,1),(02,0),(02,2),(12,1),(12,2)\}.}
\]
{Thus each row in \(S\) has exactly one available cell, namely}
\[
{(01,2),\qquad (02,1),\qquad (12,0).}
\]
{This tiny case is useful only as a notation guide: it shows concretely what we mean by rows, columns, occupied cells, and available cells, and it also illustrates that a 2-edge is simply an unordered pair of available cells. The first genuinely nontrivial optimization phenomena occur for larger values of \(q\), but this \(q=2\) picture serves as a convenient model for the general definitions below.}
\end{example}

\blue{The following proposition shows that, for these parameters, the extremal \(C_4\)-free bipartite graph is unique up to isomorphism. Consequently, the choice of the 1-edge subgraph is unambiguous in the incidence-graph family under study.}

\begin{Prop}
\blue{Let \(n=q+1\) and \(m=\binom{n}{2}\). If \(H=(X,Y,F)\) is a \(C_4\)-free bipartite graph with}
\[
\blue{|X|=m,\qquad |Y|=n,\qquad |F|=n(n-1),}
\]
\blue{then \(H\) is isomorphic to the incidence graph of \(K_n\).}
\end{Prop}

\begin{proof}
\blue{For each \(x\in X\), let \(d(x)\) denote its degree. Because \(H\) is \(C_4\)-free, any two vertices of \(Y\) have at most one common neighbor in \(X\). Hence}
\[
\blue{\sum_{x\in X}\binom{d(x)}{2}\le \binom{n}{2}=m.}
\]
\blue{On the other hand, for every integer \(d\ge 0\),}
\[
\blue{\binom{d}{2}\ge d-1,}
\]
\blue{with equality if and only if \(d\in\{1,2\}\). Therefore}
\[
\blue{m\ge \sum_{x\in X}\binom{d(x)}{2}\ge \sum_{x\in X}(d(x)-1)=\sum_{x\in X}d(x)-|X|.}
\]
\blue{Now \(\sum_{x\in X}d(x)=|F|=n(n-1)\), while \(|X|=m=\binom{n}{2}=n(n-1)/2\). Thus}
\[
\blue{\sum_{x\in X}d(x)-|X|=n(n-1)-\binom{n}{2}=m.}
\]
\blue{So equality holds throughout. In particular, \(\binom{d(x)}{2}=d(x)-1\) for every \(x\in X\), hence \(d(x)\in\{1,2\}\) for all \(x\). Since the average degree on \(X\) is}
\[
\blue{\frac{|F|}{|X|}=\frac{n(n-1)}{\binom{n}{2}}=2,}
\]
\blue{every vertex of \(X\) must actually have degree exactly \(2\).}

\blue{Each vertex \(x\in X\) therefore determines an unordered pair of neighbors in \(Y\). Two distinct vertices of \(X\) cannot determine the same pair, for otherwise those two vertices together with those two vertices of \(Y\) would form a \(C_4\). Hence the map \(x\mapsto N(x)\) is an injection from \(X\) into \(\binom{Y}{2}\). But both sets have size \(m=\binom{n}{2}\), so this injection is a bijection. Thus every two-element subset of \(Y\) occurs exactly once as the neighborhood of a vertex of \(X\). This is precisely the incidence graph of \(K_n\).}
\end{proof}

An \textbf{augmented bipartite graph} has the form
\[
G=(S,T,E_1\cup E_2),
\]
where \(E_2\) is a set of \textbf{2-edges}. A 2-edge is written as
\[
(i_1i_2,i_3; i_4i_5,i_6),
\]
where \(\{i_1,i_2\},\{i_4,i_5\}\in S\) and \(i_3,i_6\in T\). Its two halves are the cells
\[
(\{i_1,i_2\},i_3),\qquad (\{i_4,i_5\},i_6).
\]
Following \cite{QCX26a}, a 2-edge is called \textbf{nondegenerate} if
\[
\{i_1,i_2\}\neq \{i_4,i_5\},\qquad i_3\neq i_6.
\]
It is called \textbf{row-degenerate} if \(\{i_1,i_2\}=\{i_4,i_5\}\) and \(i_3\neq i_6\), and \textbf{column-degenerate} if \(\{i_1,i_2\}\neq \{i_4,i_5\}\) and \(i_3=i_6\). Both degenerate and nondegenerate 2-edges are allowed in the general definition of \(z_L\). The case \(\{i_1,i_2\}=\{i_4,i_5\}\) and \(i_3=i_6\) is excluded by the simplicity condition \((S)\).

\subsection*{Admissibility conditions}

We use the following three conditions.
\begin{enumerate}
  \item[(S)] No 2-edge shares a cell with a 1-edge or with another 2-edge.
  \item[(C2)] For a \textbf{nondegenerate} 2-edge \((i_1i_2,i_3; i_4i_5,i_6)\), the two opposite cells
  \[
  (\{i_1,i_2\},i_6),\qquad (\{i_4,i_5\},i_3)
  \]
  cannot both be occupied.
  \item[(C3)] There do not exist a 2-edge \((i_1i_2,i_3; i_4i_5,i_6)\) and an occupied cell \((x,y)\), with
  \[
  x\notin\bigl\{\{i_1,i_2\},\{i_4,i_5\}\bigr\},\qquad y\notin\{i_3,i_6\},
  \]
  such that the five cells
  \[
  (x,y),\ (x,i_3),\ (x,i_6),\ (\{i_1,i_2\},y),\ (\{i_4,i_5\},y)
  \]
  are all occupied. If the 2-edge is nondegenerate, then these five cells are additionally required to be pairwise distinct.
\end{enumerate}

\begin{Def}
For \(m=\binom{q+1}{2}\) and \(n=q+1\), consider the incidence-graph family associated with \(K_{q+1}\). The limited augmented Zarankiewicz number in this family is the maximum value of
\[
|E_1|+|E_2|
\]
over all augmented bipartite graphs \(G=(S,T,E_1\cup E_2)\) satisfying \((S)\), \((C2)\), and \((C3)\), with \(G_1\) fixed as the incidence graph of \(K_{q+1}\).
\end{Def}

By \cite{QCX26a}, every admissible construction yields a lower bound for \(\operatorname{BSR}(m,n)\):
\[
\operatorname{BSR}(m,n)\ge z_L(m,n)\ge z(m,n).
\]

{With the incidence graph, the available cells, and the admissibility conditions now fixed, we can turn to the first optimization layer: the strong framework, where the 1-edge graph is fixed and the task is to maximize the size of an admissible 2-edge family.}

\section{The strong framework: optimization model and computational procedure}

\blue{We first treat the original condition system from \cite{QCX26a}. In the present paper, condition \((C1)\) is enforced by restricting attention to nondegenerate 2-edges in the benchmark and constructive computations discussed below; the active feasibility conditions are then \((S),(C2)\), and \((C3)\).} {This part serves as the baseline for the whole paper: we first optimize in the more restrictive original framework, and only afterwards pass to the weaker global framework in Section~5.} \blue{For \(q=3\) and \(q=4\), we solve the problem exactly on the full candidate family by means of a direct 0--1 ILP whose variables record both the selected 2-edges and the occupied available cells. For \(q=5\), \(q=6\), and \(q=7\), we report explicit admissible families whose validity has been checked directly against \((S)\), \((C2)\), and \((C3)\). Throughout this section, the terms degenerate and nondegenerate are understood in the sense of \cite{QCX26a}: a 2-edge is nondegenerate precisely when its two halves lie in different rows and different columns.}

\subsection*{Available cells and candidate families}

For each row \(r\in S=\binom{V}{2}\) and each column \(c\in T=V\), the cell \((r,c)\) is already occupied by a 1-edge if and only if \(c\in r\). Hence the set of available cells is
\[
A_q:=\{(r,c)\in S\times T: c\notin r\}.
\]
Since every row \(r\in S\) excludes exactly two columns, we have
\[
|A_q|=\binom{q+1}{2}(q-1).
\]
A 2-edge is determined by an unordered pair of distinct available cells. Accordingly, the full candidate family is
\[
\mathcal E_q:=\bigl\{\{a,b\}: a,b\in A_q,\ a\neq b\bigr\}.
\]
For a candidate \(e=\{a,b\}\in\mathcal E_q\), write
\[
a=(r_1,c_1),\qquad b=(r_2,c_2).
\]
Then \(e\) is nondegenerate if and only if \(r_1\neq r_2\) and \(c_1\neq c_2\). It is row-degenerate if \(r_1=r_2\) and \(c_1\neq c_2\), and column-degenerate if \(r_1\neq r_2\) and \(c_1=c_2\).

\blue{For \(q=3\) and \(q=4\), we use the full candidate family in the exact ILP. For \(q=5\), the full candidate family has size}
\[
|\mathcal E_5|=\binom{60}{2}=1770,
\]
\blue{while its nondegenerate subfamily has size}
\[
|\mathcal E_5^{\mathrm{nd}}|=1410.
\]
\blue{For \(q=6\) and \(q=7\), we do not enumerate the full candidate families in this paper; instead we present directly verified admissible examples produced by the constructive search described below.}

\subsection*{Direct 0--1 ILP formulation}

Fix a candidate family \(\mathcal E\subseteq \mathcal E_q\). We now write the exact optimization model in a standard form that can be implemented directly. This is a \textbf{binary linear integer program} \cite{Wol98,NW99}, a class of problems known to be NP-hard in general \cite{GJ79}, but tractable at the scales considered here using standard MILP solvers.

\paragraph{Index sets and auxiliary notation.}
For each candidate \(e\in\mathcal E\), write
\[
e=\{(r_1(e),c_1(e)),(r_2(e),c_2(e))\},
\]
where \((r_1(e),c_1(e))\) and \((r_2(e),c_2(e))\) are its two halves. Let
\[
\mathcal N:=\{e\in\mathcal E: r_1(e)\neq r_2(e),\ c_1(e)\neq c_2(e)\}
\]
be the set of nondegenerate candidates.

For each available cell \(a\in A_q\) and each candidate \(e\in\mathcal E\), define the incidence coefficient
\[
M_{a,e}:=
\begin{cases}
1, & \text{if }a\text{ is one of the two halves of }e,\\
0, & \text{otherwise.}
\end{cases}
\]
For every cell \((r,c)\in S\times T\), define the occupancy symbol
\[
\bar o_{r,c}:=
\begin{cases}
1, & c\in r,\\
o_{(r,c)}, & (r,c)\in A_q.
\end{cases}
\]
Thus \(\bar o_{r,c}\) is a constant on 1-edge cells and a binary variable on available cells.

For each candidate \(e\in\mathcal E\), define the admissible witness set
\[
\mathcal W(e):=\Bigl\{(x,y): x\in S\setminus\{r_1(e),r_2(e)\},\ y\in T\setminus\{c_1(e),c_2(e)\}\Bigr\},
\]
with the additional restriction that, when \(e\in\mathcal N\), the five cells
\[
(x,y),\ (x,c_1(e)),\ (x,c_2(e)),\ (r_1(e),y),\ (r_2(e),y)
\]
must be pairwise distinct.

\paragraph{Decision variables.}
For each candidate \(e\in\mathcal E\), let
\[
x_e\in\{0,1\}
\]
indicate whether \(e\) is selected. For each available cell \(a\in A_q\), let
\[
o_a\in\{0,1\}
\]
indicate whether \(a\) is occupied by a selected 2-edge.

\paragraph{Optimization model.}
We solve
\[
\min -\sum_{e\in\mathcal E} x_e
\]
subject to
\begin{equation}
o_a-\sum_{e\in\mathcal E} M_{a,e}x_e = 0,
\qquad a\in A_q,
\tag{ILP-S}
\end{equation}
\begin{equation}
x_e+\bar o_{r_1(e),c_2(e)}+\bar o_{r_2(e),c_1(e)} \le 2,
\qquad e\in\mathcal N,
\tag{ILP-C2}
\end{equation}
and
\begin{equation}
x_e+\bar o_{x,y}+\bar o_{x,c_1(e)}+\bar o_{x,c_2(e)}+\bar o_{r_1(e),y}+\bar o_{r_2(e),y} \le 5,
\qquad e\in\mathcal E,\ (x,y)\in\mathcal W(e),
\tag{ILP-C3}
\end{equation}
together with
\[
x_e\in\{0,1\}\quad (e\in\mathcal E),
\qquad
 o_a\in\{0,1\}\quad (a\in A_q).
\]

This formulation is a \textbf{binary linear integer program}. Indeed, the objective function is linear in the variables \(x_e\). Constraint \((\mathrm{ILP}\text{-}\mathrm{S})\) is linear because each equation involves only a binary occupancy variable \(o_a\) and a linear sum of candidate-selection variables \(x_e\). Constraint \((\mathrm{ILP}\text{-}\mathrm{C2})\) is also linear, because for every cell \((r,c)\in S\times T\) the symbol \(\bar o_{r,c}\) is not a product term: it equals the constant \(1\) when \((r,c)\in E_1\), and it equals the binary variable \(o_{(r,c)}\) when \((r,c)\in A_q\). The same observation shows that \((\mathrm{ILP}\text{-}\mathrm{C3})\) is linear as well. Therefore the entire model can be implemented directly in standard MILP solvers without any linearization step.

Constraint \((\mathrm{ILP}\text{-}\mathrm{S})\) is exactly the simplicity condition \((S)\): each available cell can be used by at most one selected 2-edge, and the variable \(o_a\) records whether that cell is used. Constraint \((\mathrm{ILP}\text{-}\mathrm{C2})\) is the direct linear encoding of condition \((C2)\) for nondegenerate 2-edges. Constraint \((\mathrm{ILP}\text{-}\mathrm{C3})\) is the direct linear encoding of condition \((C3)\), with the distinctness requirement imposed through the definition of \(\mathcal W(e)\) whenever \(e\) is nondegenerate.

If the optimal objective value of the above minimization model is \(\zeta(\mathcal E)\), then the maximum number of selected 2-edges in the chosen candidate family is
\[
|E_2|_{\max}(\mathcal E)=-\zeta(\mathcal E).
\]
Hence, when \(\mathcal E=\mathcal E_q\) is the full candidate family,
\[
z_L\!\left(\binom{q+1}{2},q+1\right)=q(q+1)-\zeta(\mathcal E_q).
\]
If \(\mathcal E\) is only a restricted subfamily, then every feasible solution yields the lower bound
\[
z_L\!\left(\binom{q+1}{2},q+1\right)\ge q(q+1)+\sum_{e\in\mathcal E}x_e.
\]
\blue{This is exactly the model implemented, in equivalent programmatic form, by the checked Python and MATLAB codes.}

\subsection*{\blue{Constructive search algorithm for lower bounds}}

\blue{For \(q=5,6,7\), we use the following constructive routine to produce large admissible families efficiently. The search step is heuristic, whereas the final admissibility test is exact; consequently, every accepted output yields a rigorous lower bound.}

\begin{algorithm}[htbp]
\caption{\blue{Constructive lower-bound search with exact final verification}}
\label{alg:constructive-lb}
\begin{algorithmic}[1]
\Require \blue{An integer \(q\), a candidate family \(\mathcal E_q\) (full or nondegenerate), a time limit, and a number of randomized restarts}
\Ensure \blue{A verified admissible family \(E_2\), hence a lower bound \(z_L\ge q(q+1)+|E_2|\)}
\State \blue{Initialize \(E_2^{\mathrm{best}}\gets \emptyset\)}
\For{\blue{each randomized restart}}
  \State \blue{Generate a randomized ordering of candidates in \(\mathcal E_q\)}
  \State \blue{Initialize \(E_2\gets \emptyset\) and the occupied-cell set induced by \(E_1\)}
  \For{\blue{each candidate 2-edge \(e\) in the chosen order}}
    \If{\blue{adding \(e\) preserves \((S)\), \((C2)\), and \((C3)\)}}
      \State \blue{Insert \(e\) into \(E_2\)}
      \State \blue{Update the occupied-cell set}
    \EndIf
  \EndFor
  \State {Apply local improvement: delete one or two members of \(E_2\), refill greedily in the current candidate order, and accept the modification only if it strictly increases \(|E_2|\)}
  \If{\blue{\(\lvert E_2 \rvert > \lvert E_2^{\mathrm{best}} \rvert\)}}
    \State \blue{Set \(E_2^{\mathrm{best}}\gets E_2\)}
  \EndIf
\EndFor
\State \blue{Verify \(E_2^{\mathrm{best}}\) directly against \((S)\), \((C2)\), and \((C3)\)}
\State \blue{Output \(E_2^{\mathrm{best}}\) and the lower bound \(q(q+1)+|E_2^{\mathrm{best}}|\)}
\end{algorithmic}
\end{algorithm}

\blue{This is the procedure used for the lower-bound examples reported in Section 4.} {More precisely, each restart begins from a fresh randomized ordering of candidates; the greedy phase scans that order once and inserts every feasible candidate; the local-improvement phase then tests one-edge and two-edge deletions and refills greedily from the same order. A modified family is accepted only when its cardinality is strictly larger, and the process stops when no such improving deletion-refill move is available.} \blue{In particular, the displayed families for \(q=5,6,7\) are not recorded merely because they were found by a heuristic search; they are recorded only after an exact final admissibility verification.}

\section{Strong-framework computational results}

\blue{We now summarize the strong-framework computational results, distinguishing exact values from certified lower bounds obtained by verified feasible families.}

\begin{Thm}
In the incidence-graph family of complete graphs,
\[
z_L(6,4)=14,\qquad z_L(10,5)=26,\qquad z_L(15,6)\ge 43,
\]
and
\[
z_L(21,7)\ge 64,\qquad z_L(28,8)\ge 88.
\]
\end{Thm}

\begin{table}[htbp]
\centering
\caption{Strong-framework computational results in the incidence-graph family: exact values and certified lower bounds}
\label{tab:smallq}
\begin{tabular}{|c|c|c|c|c|}
\hline
\(q\) & Parameters & Source of the result & Result for \(|E_2|\) & Consequence for \(z_L\) \\ \hline
3 & \((6,4)\) & full candidate family, exact direct ILP & \(2\) & \(14\) \\
4 & \((10,5)\) & full candidate family, exact direct ILP & \(6\) & \(26\) \\
5 & \((15,6)\) & constructive search + direct verification & \(\ge 13\) & \(\ge 43\) \\
6 & \((21,7)\) & constructive search + direct verification & \(\ge 22\) & \(\ge 64\) \\
7 & \((28,8)\) & constructive search + direct verification & \(\ge 32\) & \(\ge 88\) \\ \hline
\end{tabular}
\end{table}

For \(q=3\), the full candidate family has size \(66\), and the exact ILP optimum is \(|E_2|=2\). Hence
\[
z_L(6,4)=12+2=14.
\]
\blue{One admissible family attaining this value is}
\[
\blue{E_2=\{(01,2;03,1),\ (13,2;23,0)\}.}
\]

For \(q=4\), the full candidate family has size \(435\), and the exact ILP optimum is \(|E_2|=6\). Hence
\[
z_L(10,5)=20+6=26.
\]
\blue{One admissible family attaining this value is}
\[
\blue{\begin{aligned}
E_2=\{&(01,4;24,3),\ (02,4;13,2),\ (03,1;12,4),\\
&(03,2;04,1),\ (12,3;34,0),\ (14,0;23,1)\}.
\end{aligned}}
\]

\begin{remark}[\blue{Auxiliary exact literal-grid benchmarks}]
\blue{Besides the incidence-graph family, we also carried out exact literal-grid computations for the small cases \((6,3)\), \((6,4)\), \((6,5)\), \((6,6)\), and \((5,5)\), under the same nondegeneracy, simplicity \((S)\), \((C2)\), and \((C3)\) rules. Here \(E_1\subseteq [m]\times[n]\) is an extremal \(C_4\)-free bipartite graph with \(|E_1|=z(m,n)\), and we then maximize \(|E_2|\) exactly over the available cells. The case \((6,4)\) coincides with the incidence-graph family for \(q=3\); the other four cases are recorded only as auxiliary directly checkable benchmarks.}

\blue{For \((6,3)\), an extremal \(C_4\)-free graph is}
\[
\blue{E_1=\{(1,1),(2,1),(3,1),(4,1),(4,2),(5,1),(5,3),(6,2),(6,3)\},}
\]
\blue{with \(|E_1|=9=z(6,3)\). Exact computation gives \(\max |E_2|=2\), attained for example by}
\[
\blue{E_2=\{((1,3),(2,2)),\ ((2,3),(3,2))\}.}
\]
\blue{Hence \(z_L(6,3)=11\).}

\blue{For \((6,4)\), an extremal \(C_4\)-free graph is}
\[
\blue{E_1=\{(1,1),(1,2),(2,1),(2,3),(3,2),(3,3),(4,1),(4,4),(5,2),(5,4),(6,3),(6,4)\},}
\]
\blue{with \(|E_1|=12=z(6,4)\). Exact computation gives \(\max |E_2|=2\), attained for example by}
\[
\blue{E_2=\{((3,1),(5,3)),\ ((4,2),(6,1))\}.}
\]
\blue{Hence \(z_L(6,4)=14\), consistent with the exact incidence-graph result above.}

\blue{For \((6,5)\), an extremal \(C_4\)-free graph is}
\[
\blue{E_1=\{(1,1),(1,2),(2,1),(2,3),(3,2),(3,4),(4,3),(4,4),(5,2),(5,3),(5,5),(6,1),(6,4),(6,5)\},}
\]
\blue{with \(|E_1|=14=z(6,5)\). Exact computation gives \(\max |E_2|=3\), attained for example by}
\[
\blue{E_2=\{((2,4),(3,5)),\ ((3,3),(5,1)),\ ((4,1),(6,2))\}.}
\]
\blue{Hence \(z_L(6,5)=17\).}

\blue{For \((6,6)\), an extremal \(C_4\)-free graph is}
\[
\blue{E_1=\{(1,1),(1,2),(2,3),(2,4),(3,1),(3,3),(3,5),(4,2),(4,4),(4,5),(5,2),(5,3),(5,6),(6,1),(6,4),(6,6)\},}
\]
\blue{with \(|E_1|=16=z(6,6)\). Exact computation gives \(\max |E_2|=4\), attained for example by}
\[
\blue{E_2=\{((1,6),(2,1)),\ ((2,5),(3,6)),\ ((3,4),(6,2)),\ ((4,3),(5,1))\}.}
\]
\blue{Hence \(z_L(6,6)=20\).}

\blue{For \((5,5)\), one extremal \(C_4\)-free graph is}
\[
\blue{E_1=\{(1,1),(1,2),(2,1),(2,3),(3,2),(3,4),(4,2),(4,3),(4,5),(5,1),(5,4),(5,5)\},}
\]
\blue{with \(|E_1|=12=z(5,5)\). Exact computation gives \(\max |E_2|=3\), attained for example by}
\[
\blue{E_2=\{((1,4),(5,3)),\ ((2,2),(4,4)),\ ((2,5),(3,1))\}.}
\]
\blue{Hence \(z_L(5,5)=15\).}

\end{remark}

For \(q=5\), we present a verified admissible family with \(|E_2|=13\). Therefore
\[
z_L(15,6)\ge 30+13=43.
\]
\blue{The family below was obtained by Algorithm~\ref{alg:constructive-lb} and then checked exactly against \((S)\), \((C2)\), and \((C3)\).}
\blue{One admissible family of size \(13\) is}
\[
\blue{\begin{aligned}
E_2=\{&(01,2;35,4),\ (01,3;45,2),\ (02,1;34,5),\ (02,3;14,5),\ (03,1;25,4),\\
&(04,3;15,2),\ (04,5;12,3),\ (05,3;24,1),\ (05,4;23,1),\ (13,0;24,5),\\
&(14,2;35,0),\ (15,4;23,0),\ (25,1;34,0)\}.
\end{aligned}}
\]
\blue{This family is nondegenerate in the sense of \cite{QCX26a}.}

\blue{For \(q=6\), a verified admissible family of size \(22\) gives}
\[
\blue{z_L(21,7)\ge 42+22=64.}
\]
\blue{The family below was obtained by Algorithm~\ref{alg:constructive-lb} and then checked exactly against \((S)\), \((C2)\), and \((C3)\).}
\blue{One such family is}
\[
\blue{\begin{aligned}
E_2=\{&(25,1;36,0),\ (04,5;16,3),\ (14,5;26,3),\ (12,5;46,0),\ (05,1;23,6),\\
&(26,4;35,1),\ (23,1;45,6),\ (03,2;15,4),\ (01,5;24,3),\ (05,4;16,2),\\
&(24,0;36,1),\ (02,4;16,5),\ (06,4;25,3),\ (06,5;13,2),\ (14,0;36,2),\\
&(34,0;56,1),\ (12,3;45,0),\ (03,4;15,6),\ (01,4;35,2),\ (34,2;56,0),\\
&(02,3;46,5),\ (13,5;26,0)\}.
\end{aligned}}
\]
\blue{This family is nondegenerate in the sense of \cite{QCX26a}.}

\blue{For \(q=7\), a verified admissible family of size \(32\) gives}
\[
\blue{z_L(28,8)\ge 56+32=88.}
\]
\blue{The family below was obtained by Algorithm~\ref{alg:constructive-lb} and then checked exactly against \((S)\), \((C2)\), and \((C3)\).}
\blue{One such family is}
\[
\blue{\begin{aligned}
E_2=\{&(24,1;36,5),\ (15,4;36,0),\ (14,5;36,7),\ (01,7;23,6),\ (13,5;24,6),\\
&(01,4;56,7),\ (23,4;67,5),\ (01,6;57,4),\ (16,3;25,4),\ (05,3;47,2),\\
&(06,3;12,7),\ (45,2;67,0),\ (04,6;35,2),\ (02,6;34,5),\ (13,7;26,4),\\
&(07,6;15,2),\ (04,2;35,7),\ (06,5;12,3),\ (27,5;46,1),\ (17,0;23,5),\\
&(03,2;17,5),\ (35,1;46,0),\ (26,0;37,1),\ (03,6;57,2),\ (16,7;45,3),\\
&(02,1;37,5),\ (12,4;56,3),\ (16,0;25,3),\ (27,6;34,1),\ (05,6;47,1),\\
&(07,2;15,3),\ (14,2;56,0)\}.
\end{aligned}}
\]
\blue{This family is nondegenerate in the sense of \cite{QCX26a}.}

\begin{remark}
Under the row/column-based definition of \cite{QCX26a}, the examples displayed above for \(q=3,4,5,6,7\) are all nondegenerate.
\end{remark}

\section{The weak framework: optimization models and computational procedures}

\blue{We next develop the weak framework, proposed in \cite{QCX26c}, from the viewpoint of global optimization and certified computation.} {Conceptually, this section relaxes the baseline strong model from Sections~3--4 and asks how much larger the admissible class becomes when the weaker rules \((S),(W2),(W2'),(W3)\) are imposed instead.} \blue{Formally, the weak admissibility rules are \((S),(W2),(W2'),(W3)\), and the associated parameter \(z_{WL}(m,n)\) is interpreted globally throughout the paper. On the literal-grid benchmarks studied later, this global weak problem is solved exactly for all fourteen cases from \((4,3)\) through \((7,7)\), allowing both nondegenerate and degenerate 2-edges as in \cite{QCX26c}. On the larger incidence-family cases, by contrast, the corresponding global exact optimization problem is presently too large to solve to optimality, so we compute rigorous lower bounds by means of tractable search-and-verification procedures. The purpose of this section is therefore to state the weak admissibility conditions and to describe the models and algorithms used to obtain the current exact values and certified lower bounds.}

\subsection*{The weak admissibility conditions}

\blue{Let \(G=(S,T,E_1\cup E_2)\) be an augmented bipartite graph with the same fixed 1-edge subgraph \(G_1=(S,T,E_1)\) as before. The weak framework used here consists of the following requirements.}
\begin{enumerate}
  \item[(S)] \blue{No 2-edge shares a cell with a 1-edge or with another 2-edge.}
  \item[(W2')] \blue{For a nondegenerate 2-edge, its two opposite cells are not both 1-edges. Equivalently, if \((r_1,c_1;r_2,c_2)\) is selected, then the two opposite cells \((r_1,c_2)\) and \((r_2,c_1)\) cannot both belong to \(E_1\).}
  \item[(W2)] \blue{The dependency relation among selected nondegenerate 2-edges contains no directed cycle, except that a directed 2-cycle formed by a pair of complementary 2-edges is allowed. Here a directed arc \(e\to e'\) is present only when both opposite cells of the selected source edge \(e\) are occupied in the augmented graph and one half of the selected target edge \(e'\) is one of those two opposite cells.}
  \item[(W3)] \blue{For any vertex-disjoint pair of edges in which at least one edge is a selected 2-edge, at least one corresponding opposite cell must remain unoccupied. For a vertex-disjoint 1-edge and 2-edge this gives the four opposite cells \(\{(r,j),(r,\ell),(i,c),(k,c)\}\) when the 1-edge is \((r,c)\) and the 2-edge is \((i,j;k,\ell)\). For two vertex-disjoint nondegenerate selected 2-edges \((i,j;k,\ell)\) and \((i',j';k',\ell')\), the opposite-cell set is the eight-cell set \(\{(i,j'),(i,\ell'),(k,j'),(k,\ell'),(i',j),(i',\ell),(k',j),(k',\ell)\}\), with repetitions suppressed, and at least one of these cells must remain unoccupied. This is the weak analogue of the forbidden generalized \(C_4\)-patterns encoded by \((C3)\).}
\end{enumerate}

\blue{Every family admissible in the strong framework is also admissible in the weak framework, but the converse need not hold. Consequently, the weak framework is suitable both for comparison experiments and for producing potentially larger lower bounds.}

\subsection*{Optimization model in the incidence-graph family}

\blue{Fix again the incidence graph of \(K_{q+1}\), with available-cell set \(A_q\) and nondegenerate candidate family \(\mathcal E_q^{\mathrm{nd}}\). For the larger incidence-family cases, we use this tractable candidate family to compute rigorous lower bounds for the global parameter \(z_{WL}\). The optimization model is again a binary ILP of the same type as in Section 3 \cite{Wol98,CCZ14}, with the constraints encoding the weak admissibility rules.}

\blue{For each candidate \(e\in \mathcal E_q^{\mathrm{nd}}\), let \(x_e\in\{0,1\}\) indicate whether \(e\) is selected, and for each available cell \(a\in A_q\) let \(o_a\in\{0,1\}\) indicate whether \(a\) is occupied by a selected 2-edge. As before, the simplicity condition is encoded by}
\[
o_a-\sum_{e\in\mathcal E_q^{\mathrm{nd}}} M_{a,e}x_e=0,
\qquad a\in A_q,
\]
\blue{together with \(o_a\in\{0,1\}\).}

\blue{Condition \((W2')\) is imposed only on the fixed 1-edge pattern: if both opposite cells of a candidate are already in \(E_1\), then that candidate is forbidden. Thus \((W2')\) is handled by deleting such candidates from the feasible family before optimization.}

\blue{To encode \((W2)\), we again introduce an integer order variable for each candidate. An order constraint is imposed only for an active arc \(e\to e'\) whose source edge \(e\) is selected, whose two opposite cells are both occupied in the augmented graph, and one half of the selected target edge \(e'\) lies in those opposite cells. For an ordinary realized dependency arc we impose a strict order increase, while for a realized arc inside a complementary pair we allow the two edges to lie on the same layer. Equivalently, the implemented verifier contracts each allowed complementary pair into a single component and then checks acyclicity of the resulting component digraph. Thus the \((W2)\) constraints are activated only for realized dependency arcs, not for every potential opposite-cell hit.}

\blue{Condition \((W3)\) is then written as a finite family of linear inequalities. For each selected 2-edge and each vertex-disjoint 1-edge, at least one of the four corresponding opposite cells must remain unoccupied. For each vertex-disjoint pair of selected nondegenerate 2-edges, the relevant opposite-cell set is the eight-cell set obtained by pairing each row of one edge with each column of the other edge, and symmetrically. Again at least one of these opposite cells must remain unoccupied. Since the occupied-state symbols are constants on \(E_1\) and binary variables on available cells, all these inequalities are linear. Consequently, this model is again solvable by a direct binary ILP on the chosen candidate family and provides certified lower bounds for the global weak parameter in the larger incidence-family cases.}

\subsection*{A companion search-and-verification algorithm}

\blue{For \(q=3\), this incidence-family model is small enough to solve exactly. For \(q=4\), the same direct ILP yields a verified feasible weak solution, although we do not claim an optimality certificate for the full global parameter \(z_{WL}(10,5)\). For \(q=5,6,7\), we use a randomized greedy search with {local delete-and-refill improvement}, followed by exact verification of the returned family against \((S),(W2),(W2'),(W3)\). The logic is the same as in Algorithm~\ref{alg:constructive-lb}: the search step is heuristic, but the final admissibility test is exact, so every accepted configuration yields a rigorous lower bound for the global weak parameter.}

\begin{algorithm}[htbp]
\caption{\blue{Constructive search for the weak framework}}
\label{alg:constructive-weak}
\begin{algorithmic}[1]
\Require \blue{An integer \(q\), the nondegenerate incidence-family candidate set \(\mathcal E_q^{\mathrm{nd}}\), a time limit, and a number of randomized restarts}
\Ensure \blue{A verified weakly admissible family \(E_2\)}
\State \blue{Delete every candidate forbidden by \((W2')\)}
\State \blue{Initialize \(E_2^{\mathrm{best}}\gets \emptyset\)}
\For{\blue{each randomized restart}}
  \State {Generate a randomized candidate order by sorting candidates according to a weak-conflict score, with smaller scores preferred and random perturbations used to break ties}
  \State \blue{Initialize \(E_2\gets\emptyset\)}
  \For{\blue{each candidate \(e\) in the chosen order}}
    \If{\blue{adding \(e\) preserves \((S),(W2),(W2'),(W3)\)}}
      \State \blue{Insert \(e\) into \(E_2\)}
    \EndIf
  \EndFor
  \State {Apply local delete-and-refill improvement: remove one selected candidate, refill greedily in the current order, and accept the new family only if it is larger}
  \If{\blue{\(\lvert E_2 \rvert > \lvert E_2^{\mathrm{best}} \rvert\)}}
    \State \blue{Set \(E_2^{\mathrm{best}}\gets E_2\)}
  \EndIf
\EndFor
\State \blue{Verify \(E_2^{\mathrm{best}}\) exactly against \((S),(W2),(W2'),(W3)\)}
\State \blue{Output the verified family and the corresponding weak-framework objective value}
\end{algorithmic}
\end{algorithm}

{In the implementation, the weak-conflict score used in the ordering step is the number of precomputed \(W3\)-type obstructions associated with the candidate, together with a small random perturbation and a mild geometric tie-break term depending on the row and column separations of its two halves. The local-improvement step deletes one currently selected edge, reruns the same greedy scan on the remaining family, and accepts the result only when it produces a strictly larger feasible family. Thus the heuristic component is fully specified, while the final admissibility check remains exact.}

\subsection*{Weak-framework computational results}

\blue{We now record the weak-framework values and verified lower bounds. In the sense of the weak limited augmented Zarankiewicz number, \(z_{WL}(m,n)\) is the maximum of \(|E_1|+|E_2|\) over all pairs \((E_1,E_2)\) that form a weakly admissible limited augmented bipartite graph on an \(m\times n\) grid and satisfy \(|E_1|=z(m,n)\). Thus the weak parameter is defined globally over the full feasible class, rather than relative to any single preassigned extremal \(C_4\)-free 1-edge graph.}

\blue{The incidence-family results are summarized in Table~\ref{tab:w3-incidence}. In this family the underlying 1-edge graph is fixed by the incidence construction, so the computational task reduces to optimizing over admissible 2-edge augmentations. For \(q=3\), the resulting global weak problem is small enough to solve exactly, yielding \(z_{WL}(6,4)=14\). For \(q=4,5,6,7\), however, the corresponding global exact models are presently too large to solve to optimality, and we therefore report the current certified lower bounds \(z_{WL}(10,5)\ge 29\), \(z_{WL}(15,6)\ge 46\), \(z_{WL}(21,7)\ge 65\), and \(z_{WL}(28,8)\ge 90\).}

\begin{table}[htbp]
\centering
\caption{\blue{Weak-framework computational results in the incidence-graph family of complete graphs}}
\label{tab:w3-incidence}
\begin{tabular}{|c|c|c|c|c|c|}
\hline
\(q\) & parameters & strong current total & weak \(|E_2|\) & \(z_{WL}(m,n)\) & status \\ \hline
3 & \((6,4)\) & 14 & 2 & 14 & exact \\
4 & \((10,5)\) & 26 & 9 & 29 & verified lower bound \\
5 & \((15,6)\) & \(\ge 43\) & 16 & 46 & verified lower bound \\
6 & \((21,7)\) & \(\ge 64\) & 23 & 65 & verified lower bound \\
7 & \((28,8)\) & \(\ge 88\) & 34 & 90 & verified lower bound \\ \hline
\end{tabular}
\end{table}

\blue{As in the strong framework, the fixed 1-edge subgraph in the incidence family is determined by}
\[
\blue{E_1=\{(e,v)\in \binom{V}{2}\times V: v\in e\},}
\]
\blue{so an incidence-family weak construction is determined by the displayed 2-edge family \(E_2\). We now record one verified weakly admissible family for each \(q=3,4,5,6,7\).}

For \(q=3\), the weak exact optimum is \(|E_2|=2\). Hence
\[
z_{WL}(6,4)=12+2=14.
\]
\blue{One weakly admissible family attaining this value is}
\[
\blue{E_2=\{(03,1;13,2),\ (12,3;13,0)\}.}
\]

For \(q=4\), a verified weakly admissible family with \(|E_2|=9\) gives
\[
z_{WL}(10,5)\ge 20+9=29.
\]
\blue{One such family is}
\[
\blue{\begin{aligned}
E_2=\{&(01,4;34,2),\ (02,3;12,4),\ (02,4;12,3),\ (03,1;04,2),\\
&(03,2;04,1),\ (13,0;14,2),\ (13,2;14,0),\ (23,0;24,1),\ (23,1;24,0)\}.
\end{aligned}}
\]

For \(q=5\), a verified weakly admissible family with \(|E_2|=16\) gives
\[
z_{WL}(15,6)\ge 30+16=46.
\]
\blue{One such family is}
\[
\blue{\begin{aligned}
E_2=\{&(35,0;45,1),\ (04,1;05,3),\ (14,3;15,2),\ (35,1;45,0),\ (34,0;35,2),\\
&(03,4;12,5),\ (02,4;13,5),\ (04,3;05,2),\ (01,4;23,5),\ (34,1;45,2),\\
&(24,3;25,0),\ (03,5;12,4),\ (01,5;23,4),\ (02,5;13,4),\ (24,0;25,3),\\
&(14,0;15,3)\}.
\end{aligned}}
\]

For \(q=6\), a verified weakly admissible family with \(|E_2|=23\) gives
\[
z_{WL}(21,7)\ge 42+23=65.
\]
\blue{One such family is}
\[
\blue{\begin{aligned}
E_2=\{&(06,5;12,4),\ (35,1;36,2),\ (14,3;15,2),\ (23,6;24,5),\ (14,2;15,3),\\
&(05,3;06,4),\ (25,0;34,1),\ (46,1;56,3),\ (01,3;02,4),\ (14,0;16,2),\\
&(03,1;04,5),\ (24,3;26,0),\ (02,6;12,5),\ (46,3;56,1),\ (03,6;04,1),\\
&(23,4;34,6),\ (04,2;26,3),\ (13,0;16,5),\ (13,5;16,0),\ (25,1;46,0),\\
&(13,4;24,0),\ (35,2;36,0),\ (15,6;45,2)\}.
\end{aligned}}
\]

For \(q=7\), a verified weakly admissible family with \(|E_2|=34\) gives
\[
z_{WL}(28,8)\ge 56+34=90.
\]
\blue{One such family is}
\[
\blue{\begin{aligned}
E_2=\{&(06,4;07,5),\ (23,5;24,6),\ (03,2;04,1),\ (05,1;06,2),\ (01,4;02,5),\\
&(56,2;57,1),\ (56,3;67,4),\ (13,7;14,6),\ (46,5;47,3),\ (25,4;26,3),\\
&(15,4;17,2),\ (12,7;15,6),\ (47,0;57,2),\ (45,2;47,1),\ (16,5;23,7),\\
&(25,7;34,6),\ (34,1;36,0),\ (07,3;14,5),\ (23,6;24,5),\ (36,2;37,4),\\
&(27,1;35,0),\ (05,3;07,1),\ (03,7;04,6),\ (01,7;12,5),\ (03,5;16,4),\\
&(35,6;45,0),\ (56,4;67,3),\ (02,3;17,4),\ (13,4;37,5),\ (47,2;57,0),\\
&(27,4;46,1),\ (34,0;36,1),\ (25,0;26,1),\ (17,3;35,2)\}.
\end{aligned}}
\]

\begin{remark}[\blue{Auxiliary exact literal-grid benchmarks for the weak framework}]
\blue{Besides the incidence-graph family, we also carried out exact literal-grid computations for the cases \((4,3)\), \((4,4)\), \((5,3)\), \((5,4)\), \((5,5)\), \((6,3)\), \((6,4)\), \((6,5)\), \((6,6)\), and \((7,3)\) through \((7,7)\). In these exact weak computations we impose simplicity \((S)\), allow both nondegenerate and degenerate 2-edges, apply \((W2)\) and \((W2')\) only to the nondegenerate selected 2-edges, and impose \((W3)\) on every vertex-disjoint pair of edges in which at least one edge is a 2-edge. The optimization is the global weak one: we maximize \(|E_1|+|E_2|\) over all extremal \(C_4\)-free 1-edge sets \(E_1\subseteq [m]\times[n]\) with \(|E_1|=z(m,n)\) and all compatible weakly admissible 2-edge families \(E_2\).}

\blue{For \((4,3)\), one extremal \(C_4\)-free graph is}
\[
\blue{E_1=\{(1,1),(1,2),(2,1),(2,3),(3,2),(4,2),(4,3)\},}
\]
\blue{with \(|E_1|=7=z(4,3)\). Exact computation gives \(\max |E_2|=1\), attained for example by}
\[
\blue{E_2=\{((3,3),(4,1))\}.}
\]
\blue{Hence \(z_{WL}(4,3)=8\).}

\blue{For \((4,4)\), one extremal \(C_4\)-free graph is}
\[
\blue{E_1=\{(1,1),(1,4),(2,1),(2,2),(3,1),(3,3),(4,2),(4,3),(4,4)\},}
\]
\blue{with \(|E_1|=9=z(4,4)\). Exact computation gives \(\max |E_2|=1\), attained for example by}
\[
\blue{E_2=\{((1,3),(3,2))\}.}
\]
\blue{Hence \(z_{WL}(4,4)=10\).}

\blue{For \((5,3)\), one extremal \(C_4\)-free graph is}
\[
\blue{E_1=\{(1,1),(2,2),(3,1),(3,3),(4,1),(4,2),(5,2),(5,3)\},}
\]
\blue{with \(|E_1|=8=z(5,3)\). Exact computation gives \(\max |E_2|=2\), attained for example by}
\[
\blue{E_2=\{((1,3),(3,2)),\ ((2,3),(5,1))\}.}
\]
\blue{Hence \(z_{WL}(5,3)=10\).}

\blue{For \((5,4)\), one extremal \(C_4\)-free graph is}
\[
\blue{E_1=\{(1,2),(1,3),(1,4),(2,1),(2,2),(3,4),(4,1),(4,4),(5,1),(5,3)\},}
\]
\blue{with \(|E_1|=10=z(5,4)\). Exact computation gives \(\max |E_2|=2\), attained for example by}
\[
\blue{E_2=\{((3,1),(4,2)),\ ((4,3),(5,2))\}.}
\]
\blue{Hence \(z_{WL}(5,4)=12\).}

\blue{For \((5,5)\), one extremal \(C_4\)-free graph is}
\[
\blue{E_1=\{(1,4),(1,5),(2,3),(2,5),(3,2),(3,5),(4,1),(4,5),(5,1),(5,2),(5,3),(5,4)\},}
\]
\blue{with \(|E_1|=12=z(5,5)\). Exact computation gives \(\max |E_2|=4\), attained for example by}
\[
\blue{E_2=\{((1,2),(4,3)),\ ((1,3),(4,2)),\ ((2,1),(3,4)),\ ((2,4),(3,1))\}.}
\]
\blue{Hence \(z_{WL}(5,5)=16\).}

\blue{For \((6,3)\), one extremal \(C_4\)-free graph is}
\[
\blue{E_1=\{(1,3),(2,3),(3,2),(4,1),(4,3),(5,2),(5,3),(6,1),(6,2)\},}
\]
\blue{with \(|E_1|=9=z(6,3)\). Exact computation gives \(\max |E_2|=3\), attained for example by}
\[
\blue{E_2=\{((1,1),(2,2)),\ ((1,2),(2,1)),\ ((3,1),(6,3))\}.}
\]
\blue{Hence \(z_{WL}(6,3)=12\).}

\blue{For \((6,4)\), one extremal \(C_4\)-free graph is}
\[
\blue{E_1=\{(1,2),(1,4),(2,3),(2,4),(3,1),(3,4),(4,1),(4,2),(5,2),(5,3),(6,1),(6,3)\},}
\]
\blue{with \(|E_1|=12=z(6,4)\). Exact computation gives \(\max |E_2|=2\), attained for example by}
\[
\blue{E_2=\{((2,2),(5,1)),\ ((4,3),(5,4))\}.}
\]
\blue{Hence \(z_{WL}(6,4)=14\).}

\blue{For \((6,5)\), one extremal \(C_4\)-free graph is}
\[
\blue{E_1=\{(1,3),(1,4),(2,1),(2,5),(3,1),(3,2),(3,3),(4,1),(4,4),(5,2),(5,4),(5,5),(6,3),(6,5)\},}
\]
\blue{with \(|E_1|=14=z(6,5)\). Exact computation gives \(\max |E_2|=3\), attained for example by}
\[
\blue{E_2=\{((1,1),(6,4)),\ ((2,3),(4,5)),\ ((5,1),(6,2))\}.}
\]
\blue{Hence \(z_{WL}(6,5)=17\).}

\blue{For \((6,6)\), one extremal \(C_4\)-free graph is}
\[
\blue{E_1=\{(1,2),(1,3),(2,1),(2,3),(2,5),(3,2),(3,5),(3,6),(4,3),(4,4),(4,6),(5,1),(5,2),(5,4),(6,1),(6,6)\},}
\]
\blue{with \(|E_1|=16=z(6,6)\). Exact computation gives \(\max |E_2|=5\), attained for example by}
\[
\blue{E_2=\{((1,1),(6,4)),\ ((1,6),(5,3)),\ ((2,2),(3,4)),\ ((2,4),(6,5)),\ ((4,5),(5,6))\}.}
\]
\blue{Hence \(z_{WL}(6,6)=21\).}

\blue{For \((7,3)\), one extremal \(C_4\)-free graph is}
\[
\blue{E_1=\{(1,3),(2,3),(3,2),(4,2),(5,1),(5,3),(6,2),(6,3),(7,1),(7,2)\},}
\]
\blue{with \(|E_1|=10=z(7,3)\). Exact computation gives \(\max |E_2|=4\), attained for example by}
\[
\blue{E_2=\{((1,1),(2,2)),\ ((1,2),(2,1)),\ ((3,1),(4,3)),\ ((3,3),(4,1))\}.}
\]
\blue{Hence \(z_{WL}(7,3)=14\).}

\blue{For \((7,4)\), one extremal \(C_4\)-free graph is}
\[
\blue{E_1=\{(1,2),(1,3),(2,3),(2,4),(3,1),(3,4),(4,1),(4,3),(5,2),(6,2),(6,4),(7,1),(7,2)\},}
\]
\blue{with \(|E_1|=13=z(7,4)\). Exact computation gives \(\max |E_2|=4\), attained for example by}
\[
\blue{E_2=\{((1,4),(6,1)),\ ((2,2),(5,1)),\ ((3,2),(7,3)),\ ((4,4),(5,3))\}.}
\]
\blue{Hence \(z_{WL}(7,4)=17\).}

\blue{For \((7,5)\), one extremal \(C_4\)-free graph is}
\[
\blue{E_1=\{(1,4),(1,5),(2,2),(2,4),(3,3),(3,5),(4,2),(4,5),(5,3),(5,4),(6,1),(6,2),(6,3),(7,1),(7,4)\},}
\]
\blue{with \(|E_1|=15=z(7,5)\). Exact computation gives \(\max |E_2|=6\), attained for example by}
\[
\blue{E_2=\{((1,2),(5,1)),\ ((2,3),(7,5)),\ ((2,5),(7,3)),\ ((3,1),(4,4)),\ ((3,4),(4,1)),\ ((5,2),(6,5))\}.}
\]
\blue{Hence \(z_{WL}(7,5)=21\).}

\blue{For \((7,6)\), one extremal \(C_4\)-free graph is}
\[
\blue{\begin{aligned}
E_1=\{&(1,4),(1,5),(2,2),(2,5),(2,6),(3,3),(3,6),(4,1),(4,4),\\
&(4,6),(5,1),(5,2),(6,1),(6,3),(6,5),(7,2),(7,3),(7,4)\},
\end{aligned}}
\]
\blue{with \(|E_1|=18=z(7,6)\). Exact computation gives \(\max |E_2|=6\), attained for example by}
\[
\blue{E_2=\{((1,3),(7,1)),\ ((1,6),(5,3)),\ ((2,1),(7,6)),\ ((3,2),(5,5)),\ ((3,4),(4,5)),\ ((4,2),(6,4))\}.}
\]
\blue{Hence \(z_{WL}(7,6)=24\).}

\blue{For \((7,7)\), one extremal \(C_4\)-free graph is}
\[
\blue{\begin{aligned}
E_1=\{&(1,5),(1,6),(1,7),(2,3),(2,4),(2,7),(3,2),(3,4),(3,6),(4,2),(4,3),\\
&(4,5),(5,1),(5,4),(5,5),(6,1),(6,3),(6,6),(7,1),(7,2),(7,7)\},
\end{aligned}}
\]
\blue{with \(|E_1|=21=z(7,7)\). Exact computation gives \(\max |E_2|=7\), attained for example by}
\[
\blue{E_2=\{((1,3),(7,6)),\ ((1,4),(5,2)),\ ((2,1),(3,7)),\ ((2,5),(4,6)),\ ((3,5),(6,2)),\ ((4,1),(7,4)),\ ((5,3),(6,7))\}.}
\]
\blue{Hence \(z_{WL}(7,7)=28\).}
\end{remark}

\blue{These computations should be interpreted conservatively. In the fourteen exact literal-grid cases, the stated weak values are genuine exact values of the global parameter \(z_{WL}\). In the larger incidence-family cases, the values come from explicitly verified constructions and therefore give rigorous lower bounds for the same global parameter, but not yet certified optima.}

\section{Computational comparison of the strong and weak frameworks}

\blue{With the strong and weak computational pipelines in place, we can now compare the two condition systems on common benchmark sets. The purpose of this section is not to introduce further admissibility rules, but to quantify the computational effect of replacing the strong constraints \((C2),(C3)\) by the weaker rules \((W2),(W2'),(W3)\). We first compare the exact literal-grid benchmarks, where the global weak problem is solved directly, and then compare the larger incidence-family computations, where the current weak values are reported as certified lower bounds because the corresponding global exact models remain computationally too large.}

\begin{table}[htbp]
\centering
\caption{\blue{Exact global literal-grid benchmarks for \((4,3)\) through \((7,7)\)}}
\label{tab:literal-benchmark-all}
\begin{tabular}{|c|c|c|c|c|c|}
\hline
\((m,n)\) & \(z(m,n)\) & strong \(|E_2|\) & strong total & weak \(|E_2|\) & \(z_{WL}(m,n)\) \\ \hline
\((4,3)\) & 7 & 1 & 8 & 1 & 8 \\
\((4,4)\) & 9 & 1 & 10 & 1 & 10 \\
\((5,3)\) & 8 & 1 & 9 & 2 & 10 \\
\((5,4)\) & 10 & 2 & 12 & 2 & 12 \\
\((5,5)\) & 12 & 3 & 15 & 4 & 16 \\
\((6,3)\) & 9 & 2 & 11 & 3 & 12 \\
\((6,4)\) & 12 & 2 & 14 & 2 & 14 \\
\((6,5)\) & 14 & 3 & 17 & 3 & 17 \\
\((6,6)\) & 16 & 4 & 20 & 5 & 21 \\
\((7,3)\) & 10 & 2 & 12 & 4 & 14 \\
\((7,4)\) & 13 & 3 & 16 & 4 & 17 \\
\((7,5)\) & 15 & 4 & 19 & 6 & 21 \\
\((7,6)\) & 18 & 5 & 23 & 6 & 24 \\
\((7,7)\) & 21 & 5 & 26 & 7 & 28 \\ \hline
\end{tabular}
\end{table}

\blue{Table~\ref{tab:literal-benchmark-all} records the exact literal-grid benchmark block now used throughout the paper. Across these fourteen exact test cases, the weak framework is never worse than the strong framework. Several cases improve under the global weak definition, while the remaining cases coincide with the strong exact value. Thus the exact computations still support the interpretation that replacing \((C2),(C3)\) by \((W2),(W2'),(W3)\) produces a genuine enlargement of the feasible class already at the smallest literal-grid scales.}

\begin{table}[htbp]
\centering
\caption{\blue{Exact literal-grid cases with strict improvement under the weak framework}}
\label{tab:w3-literal-improvements}
\begin{tabular}{|c|c|c|c|}
\hline
case & strong total & \(z_{WL}(m,n)\) & gain \\ \hline
\((5,3)\) & 9 & 10 & 1 \\
\((5,5)\) & 15 & 16 & 1 \\
\((6,3)\) & 11 & 12 & 1 \\
\((6,6)\) & 20 & 21 & 1 \\
\((7,3)\) & 12 & 14 & 2 \\
\((7,4)\) & 16 & 17 & 1 \\
\((7,5)\) & 19 & 21 & 2 \\
\((7,6)\) & 23 & 24 & 1 \\
\((7,7)\) & 26 & 28 & 2 \\ \hline
\end{tabular}
\end{table}

\blue{For the incidence-graph family, Table~\ref{tab:w3-incidence} shows an even clearer separation. The weak framework already beats the strong exact value at \(q=4\), improves the current strong lower bound at \(q=5\), and also improves the present verified strong lower bounds at \(q=6\) and \(q=7\), namely through \(z_{WL}(10,5)\ge 29\), \(z_{WL}(15,6)\ge 46\), \(z_{WL}(21,7)\ge 65\), and \(z_{WL}(28,8)\ge 90\). These values should again be read as rigorous lower bounds for the global weak parameter in the larger cases, since the corresponding exact optimization problems are presently too large to solve to optimality. Thus the computational evidence suggests that the weaker condition system is a meaningful enlargement of the feasible class rather than a merely cosmetic reformulation.}
\section{\blue{A lifting-based comparison for \((21,7)\)}}

We next record the lifting argument used to obtain the bound for \((21,7)\).

\begin{Thm}[Lifting bound, \cite{QCX26b}]
Suppose an admissible configuration for the incidence graph of \(K_{q+1}\) contains \(t\) 2-edges. Then there exists an admissible configuration for the incidence graph of \(K_{q+2}\) containing at least
\[
t+\left\lfloor \frac{q}{2}\right\rfloor
\]
2-edges.
\end{Thm}

Applying this theorem with \(q=5\) and \(t=13\) gives a configuration for \(K_7\) with at least
\[
13+\left\lfloor\frac52\right\rfloor=15
\]
2-edges. Since the incidence graph of \(K_7\) has \(42\) 1-edges, we obtain
\[
z_L(21,7)\ge 42+15=57.
\]
\blue{This lifting bound is weaker than the directly verified bound \(z_L(21,7)\ge 64\) obtained in Section 4, but it remains a useful general construction principle.}

For comparison, lifting the exact \(q=4\) value gives only
\[
z_L(15,6)\ge 30+\left(6+\left\lfloor\frac42\right\rfloor\right)=38,
\]
which is weaker than the bound \(z_L(15,6)\ge 43\) coming from the verified \(q=5\) family.

\section{Comparison with the \(K_{4t}\) benchmark construction}

For \(t\ge 1\), the \(K_{4t}\) construction of \cite{QCX26b} gives
\[
z_L\!\left(\binom{4t}{2},4t\right)\ge 2\binom{4t}{2}+4t^2-2t,
\]
so the asymptotic gap ratio is at least \(1/4\).

\blue{For the small cases \(q=4,5,6,7\), the exact value at \(q=4\) and the lower bounds obtained for \(q=5,6,7\) give the following ratios.}

\begin{table}[htbp]
\centering
\caption{Gap ratios \((z_L-z)/z\) in the small cases}
\label{tab:ratio}
\begin{tabular}{|c|c|c|c|c|}
\hline
\(q\) & \(n=q+1\) & \(z\) & Current value/bound for \(z_L\) & Gap ratio \\ \hline
4 & 5 & 20 & 26 & 30.0\% \\
5 & 6 & 30 & \(\ge 43\) & \(\ge 43.3\%\) \\
6 & 7 & 42 & \(\ge 64\) & \(\ge 52.4\%\) \\
7 & 8 & 56 & \(\ge 88\) & \(\ge 57.1\%\) \\ \hline
\end{tabular}
\end{table}

\blue{Therefore these small cases already lie above the asymptotic \(25\%\) benchmark supplied by the \(K_{4t}\) construction. Whether this remains true asymptotically is an open problem.}

\section{Conclusion}

\blue{We studied limited augmented Zarankiewicz numbers in the incidence-graph family through optimization models and certified computations under both the original framework \((C1),(S),(C2),(C3)\) and the weak framework \((S),(W2),(W2'),(W3)\). For the original framework, a direct 0--1 ILP together with constructive search and exact final admissibility verification yields the exact values \(z_L(6,4)=14\) and \(z_L(10,5)=26\), as well as the certified lower bounds \(z_L(15,6)\ge 43\), \(z_L(21,7)\ge 64\), and \(z_L(28,8)\ge 88\).}
\[
\blue{z_L(6,4)=14,\qquad z_L(10,5)=26,\qquad z_L(15,6)\ge 43,}
\]
\[
\blue{z_L(21,7)\ge 64,\qquad z_L(28,8)\ge 88.}
\]

\blue{For the global weak parameter \(z_{WL}\), the fourteen literal-grid benchmarks from \((4,3)\) through \((7,7)\) can be solved exactly, providing a common exact test bed on which the weak framework is never worse than the strong one and is strictly better in several cases. In the larger incidence-family cases, the corresponding global exact models are presently beyond current computational reach, and the paper therefore reports the certified lower bounds \(z_{WL}(6,4)=14\), \(z_{WL}(10,5)\ge 29\), \(z_{WL}(15,6)\ge 46\), \(z_{WL}(21,7)\ge 65\), and \(z_{WL}(28,8)\ge 90\). These computations already show a systematic advantage of the weak framework from \(q=4\) onward in the incidence family.}

\blue{Overall, the paper shows that the incidence-graph family supports a useful optimization testbed for studying admissible 2-edge augmentation: exact models are effective on small instances, while search combined with exact verification yields rigorous information on larger instances. The resulting values improve the corresponding classical Zarankiewicz numbers and therefore strengthen available lower bounds for \(\operatorname{BSR}(m,n)\) in this setting.}

\blue{We conclude with the following three open problems.}
\begin{enumerate}
  \item Determine the exact values of \(z_L(15,6)\), \(z_L(21,7)\), and \(z_L(28,8)\) in the incidence-graph family.
  \item Determine more precisely the role of degenerate 2-edges, in the row/column sense of \cite{QCX26a}, in optimal constructions.
  \item Compare the incidence-graph family with other extremal \(C_4\)-free families, and determine whether the asymptotic gap ratio can exceed \(1/4\).
\end{enumerate}

\bigskip
	
	\noindent\textbf{Acknowledgments}
	This work was partially supported by the Jiangsu Provincial Scientific Research Center of Applied Mathematics (Grant No. BK20233002).
	
	\medskip
	
	\noindent\textbf{Data availability}
	No datasets were used in this study.
	
	\medskip
	
	\noindent\textbf{Conflict of interest} The authors declare no conflict of interest.

\end{document}